\documentclass[10pt, english,oneside,reqno]{amsart}

\usepackage[utf8]{inputenc}
\usepackage[english]{babel}
\usepackage{hyperref}
\usepackage{color}
\usepackage{enumerate}
\usepackage{fancyhdr}

\usepackage{amsmath}
\usepackage{amscd}
\usepackage{amsfonts}
\usepackage{amsthm}
\usepackage{amstext}
\usepackage{amssymb}
\usepackage{amsbsy}
\usepackage{mathrsfs}
\usepackage{tikz-cd}
\usepackage{mathtools}
\usepackage{enumitem}

\usepackage{float}
\usepackage{caption}
\usepackage{subcaption}
\usepackage{graphicx}

\usepackage[alphabetic,abbrev]{amsrefs}

\usepackage{url}

\theoremstyle{plain}
\newtheorem{theorem}{Theorem}[section]

\newtheorem{definition}[theorem]{Definition}
\newtheorem{proposition}[theorem]{Proposition}
\newtheorem{corollary}[theorem]{Corollary}
\newtheorem{lemma}[theorem]{Lemma}

\theoremstyle{remark}

\numberwithin{equation}{section}
\numberwithin{figure}{section}

\newcommand{\eps}{\varepsilon}
\newcommand{\tr}{\mathrm{tr}}

\newcommand{\Q}{\mathbb{Q}}
\newcommand{\R}{\mathbb{R}}

\newcommand{\E}{\mathbb{E}}

\newcommand{\Prob}{\mathbb{P}}

\newcommand{\Haarof}[1]{m_{#1}}

\begin{document}
	\title[Entropy Theory for Random Walks on Lie Groups]{Entropy Theory for Random Walks on Lie Groups}
	\author[sk]{Samuel Kittle}
    \author[ck]{Constantin Kogler}

    \begin{abstract}
        We develop entropy and variance results for the product of independent identically distributed random variables on Lie groups. Our results apply to the study of stationary measures in various contexts.   
    \end{abstract}

    \email{s.kittle@ucl.ac.uk, kogler@ias.edu}

    \address{Samuel Kittle, Department of Mathematics, University College London, 25 Gordon Street, London WC1H 0AY, United Kingdom}

    \address{Constantin Kogler, Institute for Advanced Study, 1 Einstein Dr, Princeton, NJ 08540, United States of America}

    \maketitle

    \tableofcontents
	
    \section{Introduction}

    Entropy is a central tool in the study of random walks. For example, as exposed in the book by Johnson \cite{JohnsonInformationBook}, entropic methods can be used to prove the central limit theorem on $\R^d$ as well as equidistribution of random walks on compact groups. More recently, entropy results have been applied in the study of stationary measures as in the work of Hochman \cite{Hochman2014}, \cite{Hochman2017} on the dimension of self-similar measures, of Bárány-Hochman-Rapaport \cite{BaranyHochmanRapaport2019} for self-affine measures, or of Varjú \cite{Varju2019} to construct absolutely continuous Bernoulli convolutions.   
    
    The aim of this paper is to establish entropy and variance bounds for smoothings of random walks on arbitrary Lie groups. We strive to be as general as feasible to ensure applicability to various settings.  Our results generalise the entropy and variance bounds previously established in \cite{Kittle2023} for $\mathrm{SL}_2(\R)$, which were developed to construct absolutely continuous Furstenberg measures. Furthermore, the results presented here have been applied in \cite{KittleKogler2024ACSS} and \cite{KittleKogler2024Dimension} to study self-similar measures. We anticipate that this paper will contribute to the investigation of random walks of Lie group actions on manifolds in various contexts.
    
    Throughout this paper, let $G$ be a real Lie group of dimension $\ell$ and denote by $\mathfrak{g}$ the Lie algebra of $G$. We fix an inner product on $\mathfrak{g}$, inducing an associated norm $|\circ|$ and a left-invariant metric $d$ on $G$. Also, denote by $\log: G \to \mathfrak{g}$ the logarithm on $G$, which is defined in a small neighbourhood around the identity. For $g \in G$ we write $B_{\eps}(g)$ for the open $\eps$-ball around $g$ and abbreviate $B_{\eps} =  B_{\eps}(\mathrm{Id})$ for $\mathrm{Id}$ the identity in $G$.  Finally, $\Haarof{G}$ is the Haar measure on $G$ that is normalised such that $\Haarof{G}(B_{\eps}) / \Haarof{\mathfrak{g}}(\log B_{\eps}) \to 1$ as $\eps \to 0$, where $\Haarof{\mathfrak{g}}$ is the volume measure induced by the fixed inner product on $\mathfrak{g}$.

    Given a $G$-valued random variable $g$, we denote by $$H(g)$$ the Shannon entropy of $g$ when $g$ is discrete and the differential entropy when $g$ is absolutely continuous. Precise definitions and some basic results are given in section~\ref{Entropy:Basic} and section~\ref{Entropy:KL}.

    Similarly to \cite{Hochman2014}, \cite{Hochman2017}, \cite{BaranyHochmanRapaport2019} or \cite{Varju2019}, we study the entropy of a smoothing of $g$. If $s$ is a smoothing distribution independent of $g$, we can abstractly define the entropy of $g$ with respect to $s$ as $$H(g; s) = H(gs) - H(s).$$ 

    Concretely, we will choose the following smoothing functions: For given $r > 0$ and $a \geq 1$, denote by $\beta_{a,r}$ a random variable on $\mathfrak{g}$ with density function $f_{a,r}: \mathfrak{g} \to \R$ given by $$f_{a,r}(x) = \begin{cases}
    C_{a,r}e^{-\frac{|x|^2}{2r^2}} & \text{if } |x| \leq ar, \\
    0 &\text{otherwise},
    \end{cases}$$ where $C_{a,r}$ is a normalising constant to ensure that $f_{a,r}$ integrates to $1$. We furthermore set 
    \begin{equation}\label{sardef}
        s_{a,r} = \exp(\beta_{a,r})
    \end{equation}
    and then define the entropy of $g$ at scale $r > 0$ with respect to the parameter $a \geq 1$ as 
    \begin{equation}\label{Hadef}
        H_a(g;r) = H(g; s_{a,r}) = H(gs_{a,r}) - H(s_{a,r}).
    \end{equation}
    Entropy at scale $r$ measures the amount of information $g$ has at scale $r$. The parameter $a$ is useful in order to uniformly control how far $\beta_{a,r}$ is from a normal distribution. The reason we are working with these smoothing functions on $G$ is to deduce in 
    Lemma~\ref{VarEntProperties} a quantitative analogue of \eqref{IntroGVarIneq}. It is assumed throughout this paper that the collection of random variables $g$ and $(s_{a,r})_{a \geq 1, r > 0}$ are independent. 

    Let $\mu$ be a probability measure on $G$ and let $\gamma_1, \gamma_2, \ldots$ be independent $\mu$-distributed random variables that are independent from $(s_{a,r})_{a \geq 1, r > 0}$ and denote for $n\geq 1$, $$q_n = \gamma_1 \cdots \gamma_n.$$ Our first goal is to give a general result on the behaviour of $H_a(q_n; r_n)$ for suitable scales $r_n > 0$. To do so, for a finitely supported probability measure $\mu$ on $G$ we define the separation rate as $$M_n = \min\left\{ d(g,h) \,:\, g,h \in \bigcup_{i = 0}^n \mathrm{supp}(\mu^{*i}) \text{ with } g\neq h \right\}.$$ Furthermore, the random walk entropy of $\mu$ is given as $$h_{\mu} = \lim_{n \to \infty} \frac{1}{n}H(\mu^{*n}) = \inf_{n \geq 1} \frac{1}{n}
    H(\mu^{*n}).$$

    We first make the following basic observation. If $r_n < \frac{1}{2a}M_n$, then all the elements in the support of $q_n$ are separated by at least $2ar_n$. Therefore, the density of $q_n s_{a, r_n}$ can be expressed as the weighted sum of the densities of $xs_{a, r_n}$ with $x$ ranging over the elements in the support of $q_n$. Thus it follows (by Lemma~\ref{EntProd}) that
    \begin{equation}\label{VeryBasicEntBound}
        H_a(q_n; r_n) = H(q_n) \geq n h_{\mu}.
    \end{equation}
    Our first result is a generalisation of \eqref{VeryBasicEntBound} for arbitrary stopped random walks $q_{\eta_n}$ for a sequence of stopping times $\eta_n$. In order to deduce an analogue of \eqref{VeryBasicEntBound}, we require that our stopping times satisfy a large deviation principle. 

\begin{definition}\label{EtaLDPDef}
    Let $\eta = (\eta_n)_{n \geq 1}$ be a sequence of stopping times. Then we say that $\eta$ satisfies the large deviation  principle if $\E[\eta_n] \to \infty$ as $n \to \infty$ and for every $\eps > 0$ there exists a $\delta > 0$ such that for all sufficiently large $n$, $$\mathbb{P}\big[|\eta_n - \mathbb{E}[  \eta_n]| \geq \eps \cdot  \mathbb{E}[ \eta_n] \big] \leq e^{-\delta \cdot \mathbb{E}[\eta_n]}.$$
\end{definition}

To state the first theorem, we note that we use the asymptotic notation as explained at the end of the introduction. 

\begin{theorem}\label{StoppedRWHighEnt} Let $\mu$ be a finitely supported probability measure on $G$. Let $\eta = (\eta_n)_{n \geq 1}$ be a sequence of stopping times satisfying the large deviation principle and write $L_n = \mathbb{E}[\eta_n]$ for $n\geq 1$. Let $a \geq 1$, $\eps > 0$ and let $r_n > 0$ be a sequence satisfying for all $n\geq 1$, $$r_n \leq a^{-1}c_G M_{\lceil (1 + \eps)L_n \rceil}$$ for a constant $c_G > 0$ depending only on $G$. Then for all $n \geq 1$, 
	$$H_a(q_{\eta_{n}} ; r_n) \geq h_{\mu}\cdot L_n + o_{\mu,\eta,\eps}(L_n).$$
\end{theorem}

Denote $$S_n = -\frac{1}{n} \log M_n \quad\quad \text{ as well as } \quad\quad S_{\mu} = \limsup_{n \to \infty} S_n.$$ In numerous concrete cases $S_{\mu}$ is finite and can be bounded efficiently. Indeed, if $G$ is a linear group and all of the entries of elements in the support of $\mu$ lie in a number field $K$ and have logarithmic height at most $L$, then  $S_{\mu} \ll_{G} L\cdot [K:\Q]$ as shown in \cite{KittleKogler2024ACSS}*{Proposition 8.10}. Under the additional assumption that $S_{\mu} < \infty$, the following corollary can be deduced from Theorem~\ref{StoppedRWHighEnt}.

\begin{corollary}\label{CorStoppedRWHighEnt} Let $\mu$ be a finitely supported probability measure on $G$ and assume that $S_{\mu} < \infty$. Let $\eta = (\eta_n)_{n \geq 1}$ be a sequence of stopping times satisfying the large deviation principle and write $L_n = \mathbb{E}[\eta_n]$ for $n\geq 1$. Suppose that $a \geq 1$ and $S > S_{\mu}$. Then for any sequence $0 < r_n < e^{-S \cdot L_n}$ as $n \to \infty$,
    $$H_a(q_{\eta_{n}} ; r_n) \geq h_{\mu}\cdot L_n + o_{\mu,\eta, a, S}(L_n).$$
\end{corollary}

To provide some further context, we discuss the setting from \cite{KittleKogler2024ACSS} and \cite{KittleKogler2024Dimension} where Corollary~\ref{CorStoppedRWHighEnt} is used. A similarity is a map $g:\R^d \to \R^d$ such that there exists a scalar $\rho(g) > 0$, an orthogonal matrix $U(g) \in O(d)$ and a vector $b(g) \in \R^d$ such that $g(x) = \rho(g)U(g)x + b(g)$ for all $x \in \R^d$. Denote by $G = \mathrm{Sim}(\R^d)$ the group of similarities and let $\mu$ be a probability measure on $\R^d$ supported on finitely many contractive similarities, that is, similarities $g$ satisfying $\rho(g) \in (0,1)$. Then denote for $\kappa > 0$ by $\eta_{\kappa}$ the stopping time $\eta_{\kappa} = \inf\{ n \geq 1 \,:\, \rho(q_n) < \kappa \}$. 

As shown in \cite{KittleKogler2024ACSS}*{Lemma 3.9}, $\eta_{\kappa}$ satisfies a large deviation principle with $L_n = \mathbb{E}[\eta_n] = \frac{\log \kappa^{-1}}{|\chi_{\mu}|} + o_{\mu}(\log \kappa^{-1})$, for the logarithmic contraction rate $\chi_{\mu} = \int \log \rho(g) \, d\mu(g)$. Under the assumption that $S_{\mu} < \infty$, which is satisfied as discussed above when all of the entries of elements in the support of $\mu$ are algebraic, it therefore follows by Corollary~\ref{CorStoppedRWHighEnt} that for $S > S_{\mu}$ and $0 < r_n < e^{-S\cdot L_n}$, \begin{equation}\label{CorApplication}
H_a(q_{\eta_{\kappa}}; r_n) \geq (\log \kappa^{-1})\frac{h_{\mu}}{|\chi_{\mu}|} + o_{\mu,\eta, a, S}(\log \kappa^{-1}).\end{equation}

Results similar to \eqref{CorApplication} are an important component in the proof of the main theorems of \cite{KittleKogler2024ACSS} and \cite{KittleKogler2024Dimension}. In fact, as discussed below, we will convert estimates such as \eqref{CorApplication} into variance estimates. In \cite{KittleKogler2024ACSS} and \cite{KittleKogler2024Dimension}, these quantitative estimates will eventually allow us to apply Berry-Esseen type results to establish absolute continuity or full dimension of stationary measures. A similar approach is used in \cite{Kittle2023} to study the Furstenberg measure of $\mathrm{SL}_2(\R)$. In this paper, we provide a unified perspective on these results that can be applied to any Lie group.

We next address the second goal of this paper, which is to convert entropy estimates into variance estimates on arbitrary Lie groups. Indeed, the reader may recall that if $Z$ is an absolutely continuous random variable on $\R$ with variance $\sigma^2$ then 
\begin{equation}\label{1VarEntInequality}
    H(Z) \leq \frac{1}{2} \log(2\pi e \sigma^2),
\end{equation} where $H(Z)$ is the differential entropy of $Z$
and equality holds in \eqref{1VarEntInequality} if and only if $Z$ is distributed like a centred Gaussian with variance $\sigma^2$. We will prove an analogue of this fact on Lie groups. To do so, for random variables $g$ that are supported within small balls of a given point $g_0$ we consider the covariance matrix of the Lie group logarithm applied to $g_0^{-1}g$. This viewpoint allows us to apply a higher dimensional analogue of \eqref{1VarEntInequality} to deduce an analogous result on $G$. 

For an $\ell$-dimensional random variable $X$, we denote by $\mathrm{tr}(X)$ the trace of the covariance matrix of $X$.  In particular, we use the following definition. Given $g_0 \in G$ and a random variable $g$ on $G$ we define $$\mathrm{tr}_{g_0}(g) = \mathrm{tr}(\log(g_0^{-1}g)),$$ whenever $\log(g_0^{-1}g)$ is defined. The analogue of \eqref{1VarEntInequality}, which will be proved in Proposition~\ref{GVarianceEntIneq}, then amounts to 
\begin{equation}\label{IntroGVarIneq}
    H(g) \leq \frac{\ell}{2}\log\left( \frac{2 \pi e}{\ell} \cdot \tr_{g_0}(g) \right) + O_G(\eps)
\end{equation}
for random variables $g$ supported on $B_{\eps}(g_0)$ and $\eps > 0$ sufficiently small. 

To make the latter more useful, we can relate a certain notion of trace to entropy between scales. One defines the entropy between scales $r_1, r_2 > 0$ as 
\begin{align*}\label{DefEntScales}
    H_a(g;r_1|r_2) &= H(g; s_{r_1,a}| s_{r_2,a}) = H_a(g;r_1) - H_a(g;r_2) \\ &= (H(gs_{r_1,a}) - H(s_{r_1,a})) - ( H(gs_{r_2,a}) - H(s_{r_2,a})).
\end{align*} Roughly speaking, $H_a(g; r_1|r_2)$ measures how much more information $g$ has on scale $r_1$ than it has on scale $r_2$.

We furthermore define $\tr(g; r)$  to be the supremum of all $t\geq 0$ such that we can find some $\sigma$-algebra $\mathscr{A}$ and some $\mathscr{A}$-measurable random variable $h$ taking values in $G$ such that $$|\log(h^{-1}g)| \leq r \quad \text{ and } \quad \E[\tr_h(g|\mathscr{A})] \geq t\cdot r^2.$$ The following general result can be deduced. 

\begin{theorem}\label{MainEntropyIncrease}
    Let $g$ be a random variable taking values in $G$, let $a \geq 1$ and $r > 0$ be such that $ar$ is sufficiently small in terms of $G$ and assume that $g, s_{a,r}$ and $s_{a,2r}$ are independent random variables. Then $$\tr(g; 2ar) \gg_G a^{-2}(H_a(g; r|2r) - O_{G}(e^{-a^2/4} + a^3r)),$$ for the implied constants depending only on $G$.
\end{theorem}

The conditional distribution $(g|\mathscr{A})$ becomes relevant as we can relate the entropy between scales $H_a(g; r|2r)$ to the conditional entropy $H(g s_{a,r_1}| g s_{a,r_1})$ as shown in Lemma~\ref{BasicEntropyGrowth} and Proposition~\ref{AbstractEntropyIncrease}.

In our applications (\cite{Kittle2023}, \cite{KittleKogler2024ACSS}, \cite{KittleKogler2024Dimension}) it is desirable to convert an entropy gap into lower bounds of $\sum_{i=1}^m \tr (g ; s_i)$ for a sequence of scales $s_1, \dots, s_{m} \in (r_1, r_2)$. It will be also useful to assume that our scales satisfy $s_{i + 1} \geq A s_i$ for some $A > 0$. Indeed, the latter condition is necessary to make the variance summation method (\cite{Kittle2023}, \cite{KittleKogler2024ACSS}, \cite{KittleKogler2024Dimension}) applicable, that is to combine variance bounds of stopped random walks on different scales.

Using Theorem~\ref{MainEntropyIncrease}, we deduce the following general proposition applied in \cite{KittleKogler2024ACSS} and \cite{KittleKogler2024Dimension}. 

\begin{proposition} \label{EntropyGapToTrSum}
    Let $g$ be a $G$-valued random variable independent of $(s_{a,r})_{a \geq 1, r > 0}$ and let $0 < r_1 < r_2$. Let $a \geq 1$ such that $ar_2$ is sufficiently small in terms of $G$. Suppose that for all $r_1' \in [r_1, 2r_1]$ as well as $r_2' \in [r_2/2, 2r_2]$ it holds for some constant $C> 0$ that $$H_a(g; r_1'| r_2') \geq C.$$ Let $A > 1$. Then there exists $s_1, \dots, s_{m} \in (ar_1, 4ar_2)$ where $m = \lceil \frac{\log 4ar_2 - \log ar_1}{2\log A} \rceil$ such that for $N = \left\lceil \frac{\log r_2 - \log r_1}{\log 2} \right\rceil - 1$,
    \begin{equation*}
        \sum_{i=1}^m \tr (g ; s_i) \gg_G \frac{C - N\cdot O_G(e^{-\frac{a^2}{4}} + a^3r_2)}{a^2 \log A}
    \end{equation*}
    and $s_{i+1} \geq A s_i$ for all $1 \leq i \leq m-1$.
\end{proposition}

    We comment on the structure of this paper. After discussing basic properties of entropy in section~\ref{EntropySection}, we prove Theorem~\ref{StoppedRWHighEnt} and Corollary~\ref{CorStoppedRWHighEnt} in section~\ref{EntropyGrowth}. In section~\ref{EntropyTraceSection} we show \eqref{IntroGVarIneq} and Theorem~\ref{MainEntropyIncrease}. Finally, we prove Proposition~\ref{EntropyGapToTrSum} in section~\ref{EntGapTrSumSection}.

    \subsection*{Notation}

    We use the asymptotic notation $A \ll B$ or $A = O(B)$ to denote that $|A| \leq CB$ for a constant $C > 0$. If the constant $C$ depends on additional parameters we add subscripts. Moreover, $A \asymp B$ denotes $A \ll B$ and $B \ll A$. For a sequence $A_n$ and $B_n$ we write $A_n = o(B_n)$ to denote that $A_n/B_n \to 0$ as $n \to \infty$ and we add subscripts to indicate that the speed of convergence depends on certain parameters. 

    We write $[n] = \{ 1, \ldots , n \}$.

    It is assumed throughout this paper that the random variables $(s_{a,r})_{a \geq 1, r > 0}$ as defined in \eqref{sardef} are independent. Whenever $g$ will denote a $G$-valued random variable, it will be assumed that $g$ is independent of $(s_{a,r})_{a \geq 1, r > 0}$. Given a probability measure $\mu$, we sample independent $\mu$-distributed random variables $\gamma_1, \gamma_2, \ldots $ that are independent from $(s_{a,r})_{a \geq 1, r > 0}$.

    \subsection*{Acknowledgement} The first-named author gratefully acknowledges support from the Heilbronn Institute for Mathematical Research. This work is part of the second-named author's PhD thesis conducted at the University of Oxford, supported by a Mathematical Institute Scholarship. We thank the anonymous referee for helpful comments and corrections.

    \section{Basic Properties of Entropy on Lie groups} \label{EntropySection}

In section~\ref{Entropy:Basic} we give definitions and discuss basic properties of entropy on $G$, after which we discuss the Kullback-Leibler divergence on a general measurable space $X$ in section~\ref{Entropy:KL}. In section~\ref{RegConSection} we review regular conditional distributions in order to study conditional entropy and conditional trace in section~\ref{FirstEntropy:Conditional} and section~\ref{EntropyTraceSection}.

\subsection{Entropy and Basic Properties}\label{Entropy:Basic}

For notational convenience, we denote for $x \in [0,\infty)$ by $$h(x) = \begin{cases}  -x\log(x) &\text{if } x > 0 \\ 0 & \text{if } x =0\end{cases}$$ and recall that $h$ is concave. If $\lambda = \sum_{i} p_i \delta_{g_i}$ is a discrete probability measure on $G$, we define the Shannon entropy of $\lambda$ as $$H(\lambda) = \sum_{i} h(p_i).$$ On the other hand, given an absolutely continuous probability measure $\lambda$ on $G$ with density $f_{\lambda}$  we define 
$$H(\lambda) =  \int h(f_{\lambda})  \, d\Haarof{G}.$$ We extend the definition to finite positive measures $\lambda$ that are either absolutely continuous or discrete by setting $$H(\lambda) = ||\lambda||_1 H(\lambda/||\lambda||_1), \quad \quad \text{ where} \quad \quad ||\lambda||_1 = \lambda(G).$$ In this subsection we collect some useful basic properties of entropy. 

\begin{lemma}\label{EntSum1}
    Let $\lambda_1, \ldots , \lambda_n$ be absolutely continuous finite measures on $G$. Then $$H(\lambda_1 + \ldots + \lambda_n) \geq H(\lambda_1) + \ldots + H(\lambda_n).$$
\end{lemma}

\begin{proof}
    It suffices to prove the claim for $n = 2$. Let $f_1$ and $f_2$ be the densities of $\lambda_1$ and $\lambda_2$. Then since $h$ is concave 
    \begin{align*}
        H(\lambda_1 + \lambda_2) &= (||\lambda_1||_1 + ||\lambda_2||_1) \int h\left( \frac{f_1 + f_2}{||\lambda_1||_1 + ||\lambda_2||_1}   \right) \, d\Haarof{G} \\
        &\geq (||\lambda_1||_1 + ||\lambda_2||_1) \int \frac{||\lambda_1||_1}{||\lambda_1||_1 + ||\lambda_2||_1}h\left( \frac{f_1}{||\lambda_1||_1} \right) d\Haarof{G} \\ &+ (||\lambda_1||_1 + ||\lambda_2||_1) \int \frac{||\lambda_2||_1}{||\lambda_1||_1 + ||\lambda_2||_1}h\left( \frac{f_2}{||\lambda_2||_1} \right) d\Haarof{G} \\
        &= H(\lambda_1) + H(\lambda_2).
    \end{align*}
\end{proof}

\begin{lemma}\label{EntSum2}
    Let $p = (p_1,p_2,\ldots)$ be a probability vector and let $\lambda_1, \lambda_2 , \ldots $ be probability measures on $G$ either all absolutely continuous measures or all discrete measures with finite entropy such that $||\lambda_i||_1 = p_i$. Then 
    \begin{align*}
        H\left( \sum_{i = 1}^{\infty}  \lambda_i \right) &\leq  H(p) + \sum_{i = 1}^{\infty}  H(\lambda_i).
    \end{align*} In particular, if $p_i = 0$ for all $i > k$ for some $k \geq 1$ then $$H\left( \sum_{i = 1}^k \lambda_i \right) \leq  \log k +  \sum_{i = 1}^k  H(\lambda_i).$$
\end{lemma}

\begin{proof}
    We only consider the case of absolutely continuous measures as the proof is analogous in the discrete case. Denote the densities of $\lambda_i$ by $f_i$. Note that $h(\sum_{i = 1}^{\infty} a_i) \leq \sum_{i = 1}^{\infty} h(a_i)$ for any $a_1, a_2 ,\ldots \geq 0$. Therefore 
    \begin{align*}
        H\left( \sum_{i = 1}^{\infty}  \lambda_i \right) &= \int h\left( \sum_{i = 1}^{\infty} f_i \right) \, d\Haarof{G} \\
        &\leq \sum_{i = 1}^\infty \int h(f_i) \, d\Haarof{G} \\
        &= \sum_{i = 1}^\infty \int (-f_i(x)\log(p_{i}^{-1}f_i) -f_i(x)\log(p_i)) \, d\Haarof{G} \\
        &=  \sum_{i = 1}^\infty \int p_i h(p_{i}^{-1}f_i) d\Haarof{G} + h(p_i) \\
        &= H(p) + \sum_{i = 1}^\infty H(\lambda_i). 
    \end{align*} 
\end{proof}

\begin{lemma}\label{EntProd}
    Let $\lambda_1$ be a discrete and $\lambda_2$ be an absolutely continuous probability measure on $G$. Then $$H(\lambda_1 * \lambda_2) \leq H(\lambda_1) + H(\lambda_2)$$ Suppose further that $\lambda_1$ is supported on finitely many points with separation at least $2r$ and that the support of $\lambda_2$ is contained in a ball of radius $r$. Then
    $$H(\lambda_1 * \lambda_2) = H(\lambda_1) + H(\lambda_2).$$
\end{lemma}

\begin{proof}
    Write $\lambda_1 = \sum_{i = 1}^n p_i \delta_{g_i}$ and let $f$ be the density of $\lambda_2$. Then the density of $\lambda_1 * \lambda_2$ is given by $\sum_{i = 1}^n p_i \, f \circ g_i^{-1}$. As $h(\sum_{i = 1}^n a_i) \leq \sum_{i = 1}^n h(a_i)$ for any $a_1, \ldots , a_n \geq 0$, 
    \begin{align*}
        H(\lambda_1 * \lambda_2) &= \int h\left(\sum_{i = 1}^n p_i \, f \circ g_i^{-1}\right) \, d\Haarof{G} \\ &\leq \sum_{i = 1}^n \int h(p_i \, f \circ g_i^{-1}) \, d\Haarof{G} \\
        &= \sum_{i = 1}^n \int (p_i \, f \circ g_i^{-1}) (\log(p_i) + \log( f \circ g_i^{-1})) \, d\Haarof{G} \\
        &= H(\lambda_1) + H(\lambda_2).
    \end{align*} If $\lambda_1$ is supported on finitely many points with separation at least $2r$ and that the support of $\lambda_2$ is contained in a ball of radius $r$, then the support of the functions $f \circ g_i^{-1}$ is disjoint and the inequality in the second line is an equality.
\end{proof}

\subsection{Kullback-Leibler Divergence} \label{Entropy:KL} In this section we discuss Kullback-Leibler divergence on general measurable spaces $X$. If $\nu \ll \mu$ are measures on $X$, then we define the Kullback-Leibler divergence as $$D_{\mathrm{KL}}(\nu \, || \, \mu) = -\int \log \frac{d\nu}{d\mu} \, d\nu.$$ Observe that if $\nu$ is absolutely continuous on $G$ with respect to the Haar measure $\Haarof{G}$, then $H(\nu) = D_{\mathrm{KL}}(\nu \, || \, \Haarof{G})$. We collect some basic results on the Kullback-Leibler divergence on $X$.

\begin{lemma}\label{KLSupportBound}
    Let $\nu \ll \mu$ be measures on $G$ and assume that $\nu$ is a probability measure supported on a set $A$ of positive $\mu$ measure. Then $$D_{\mathrm{KL}}(\nu \, || \, \mu) \leq \log(\mu(A)).$$
\end{lemma}

\begin{proof}
    For convenience write $d\nu = f_{\nu}\, d\mu$. Then by Jensen's inequality,
    \begin{align*}
        D_{\mathrm{KL}}(\nu \, || \, \mu) = \int_A h\left(f_{\nu} \frac{\mu(A)}{\mu(A)} \right) \, d\mu = \int h(f_{\nu}\mu(A)) \frac{1_A}{\mu(A)} \, d\mu + \log(\mu(A)) \leq  \log(\mu(A)).
    \end{align*}
\end{proof}

\begin{lemma}\label{EntProductBound}
    Assume that we can write $X = X_1 \times \ldots \times X_m$ as a product of sub-manifolds $X_i\subset X$ and assume that $\Haarof{X} = \Haarof{X_1} \times \ldots \times \Haarof{X_m}$ for a measure $\Haarof{X}$ on $X$ and measures $\Haarof{X_i}$ on $X_i$. Let $\mu$ be a probability measure on $X$ with $\mu \ll \Haarof{G}$. Denote by $\pi_i$ the projection from $X$ to $X_i$ and by $\pi_i \mu$ the pushforward of $\mu$ under $\pi_i$. Then $$D_{\mathrm{KL}}(\mu \, || \, \Haarof{X}) \leq D_{\mathrm{KL}}(\pi_1\mu \,||\, \Haarof{X_1}) + \ldots + D_{\mathrm{KL}}(\pi_m\mu \,||\, \Haarof{X_m}).$$
\end{lemma}

\begin{proof}
    It suffices to prove the claim for $m = 2$. Denote by $f_{\mu}$ the density of $\mu$ with respect to $\Haarof{X}$ and write $$f_{\mu}^1(x_2) = \int f_{\mu}(x_1,x_2) \, d\Haarof{X_1}(x_1)  \quad \text{ and } \quad f^2_{\mu}(x_1) = \int f_{\mu}(x_1,x_2) \, d\Haarof{X_2}(x_2).$$  Therefore, 
    \begin{align*}
        D_{\mathrm{KL}}(\mu \, || \, \Haarof{X}) &= \int\int h(f_{\mu}(x_1,x_2)) \, d\Haarof{X_1}(x_1)d\Haarof{X_2}(x_2) \\
        &= \int\int h\left(\frac{f_{\mu}(x_1,x_2)}{f_{\mu}^2(x_1)}f_{\mu}^2(x_1)\right) \, d\Haarof{X_1}(x_1)d\Haarof{X_2}(x_2) \\
        &= \int\int h\left(\frac{f_{\mu}(x_1,x_2)}{f_{\mu}^2(x_1)}\right) f_{\mu}^2(x_1) \, d\Haarof{X_1}(x_1)d\Haarof{X_2}(x_2) \\
        &+ \int\int -\log(f_{\mu}^2(x_1)) f_{\mu}(x_1,x_2) \, d\Haarof{X_1}(x_1)d\Haarof{X_2}(x_2) \\
        &\leq \int h(f_{\mu}^1(x_2)) \, d\Haarof{X_2}(x_2) + \int h(f_{\mu}^2(x_1)) \, d\Haarof{X_1}(x_1) \\
        &= D_{\mathrm{KL}}(\pi_1\mu \,||\, \Haarof{X_1}) + D_{\mathrm{KL}}(\pi_2\mu \,||\, \Haarof{X_2}),
    \end{align*} having used that $h$ is concave and Jensen's inequality in the penultimate line.
\end{proof}

\begin{lemma}\label{EntProductBound2}
    Let $(X, \Haarof{X})$ and $(Y, \Haarof{Y})$ be a locally compact Hausdorff space endowed with Radon measures, and let $\Phi: X \to Y$ be a homeomorphism with $\Phi_{*}\Haarof{X} = \Haarof{Y}$. Then for a measure $\nu \ll \Haarof{X}$ on $X$ it holds that $$D_{\mathrm{KL}}(\Phi_{*}\nu || m_Y) = D_{\mathrm{KL}}(\nu || m_X).$$ 
\end{lemma}

\begin{proof} Let $f:Y \to \R$ be a continuous compactly supported function. Then $$\int f \, d\Phi_{*}\nu =  \int (f\circ \Phi) \, d\nu = \int (f\circ \Phi) \frac{d\nu}{d\Haarof{X}} \, d\Haarof{X} $$ as well as 
\begin{align*}
    \int f \, d\Phi_{*}\nu &= \int f \, \frac{d\Phi_{*}\nu}{d\Haarof{Y}}  \, d\Haarof{Y} \\ &= \int f \, \frac{d\Phi_{*}\nu}{d\Haarof{Y}} \, d\Phi_*\Haarof{X} = \int (f\circ \Phi) \left(\frac{d\Phi_{*}\nu}{d\Haarof{Y}} \circ \Phi\right) \, d\Haarof{X}.
\end{align*}
    Since $\Phi$ is a homeomorphism, every continuous compactly supported function $g: X \to \R$ can be written as $g = f \circ \Phi$ for $f:Y \to \R$ a continuous compactly supported function. Therefore, using the Riesz Representation Theorem, it holds that $m_X$-almost surely $$\frac{d\Phi_{*}\nu}{dm_Y}\circ \Phi = \frac{d\nu}{dm_X}$$ and thus
    \begin{align*}
        D_{\mathrm{KL}}(\Phi_{*}\nu || m_Y) &= -\int \log \frac{d\Phi_{*}\nu}{dm_Y} \, d\Phi_{*}\nu \\
        &= -\int \log\left( \frac{d\Phi_{*}\nu}{dm_Y}\circ \Phi \right)  \, d\nu \\
        &= -\int \log\left(\frac{d\nu}{dm_X} \right) \, d\nu = D_{\mathrm{KL}}(\nu || m_X).
    \end{align*}
\end{proof}

\begin{lemma}\label{KLBound}
    Let $\lambda_1$ be a probability measure on $X$ and let $\lambda_2$ and $\lambda_3$ be measures on $X$ such that $\lambda_1 \ll \lambda_2$ and $\lambda_2 \ll \lambda_3$. Let $U \subset X$ and suppose that the support of $\lambda_1$ is contained in $U$. Then $$|D_{\mathrm{KL}}(\lambda_1 \,||\, \lambda_2) - D_{\mathrm{KL}}(\lambda_1 \,||\, \lambda_3)| \leq \sup_{x \in U} \bigg| \log \frac{d\lambda_2}{d\lambda_3}  \bigg|.$$
\end{lemma}

\begin{proof}
    We calculate 
    \begin{align*}
        |D_{\mathrm{KL}}(\lambda_1 \,||\, \lambda_2) - D_{\mathrm{KL}}(\lambda_1 \,||\, \lambda_3)| &= \bigg|  \int_U \log \frac{d\lambda_1}{d\lambda_2} \, d\lambda_1 - \int_U \log \frac{d\lambda_1}{d\lambda_3} \, d\lambda_1  \bigg|  \\
        &\leq \int_U \bigg| \log \frac{d\lambda_1}{d\lambda_2} - \log \frac{d\lambda_1}{d\lambda_3} \bigg| \, d\lambda_1 \\
        &= \int_U \bigg|  \log \frac{d\lambda_2}{d\lambda_3} \bigg| \, d\lambda_1 \\
        &\leq \sup_{x \in U} \bigg| \log \frac{d\lambda_2}{d\lambda_3}  \bigg|.
    \end{align*}
\end{proof}

\subsection{Regular Conditional Distributions} \label{RegConSection}

In this section we review the definition of regular conditional distributions that will be used in the following subsections in order to discuss conditional entropy. On a probability space $(\Omega, \mathscr{F}, \mathbb{P})$, we denote the conditional expectation by $\E[f|\mathscr{A}]$ for $f \in L^1(\Omega, \mathscr{F}, \mathbb{P})$ and a $\sigma$-algebra $\mathscr{A} \subset \mathscr{F}$. Given two measurable spaces $(\Omega_1, \mathscr{A}_1)$ and $(\Omega_2, \mathscr{A}_2)$, recall that a Markov kernel on $(\Omega_1, \mathscr{A}_1)$ and $(\Omega_2, \mathscr{A}_2)$ is a map $\kappa : \Omega_1 \times \mathscr{A}_2 \to [0,1]$ if for any $A_2 \in \mathscr{A}_2$, the map $\kappa(\cdot, A_2)$ is $\mathscr{A}_1$-measurable and for any $\omega_1$ the map $A_2 \to \kappa(\omega_1, A_2)$ is a probability measure. 

\begin{definition}
    Let $(\Omega, \mathscr{F}, \mathbb{P})$ be a probability space and let $\mathscr{A} \subset \mathscr{F}$ be a $\sigma$-algebra. Let $(E, \xi)$ be a measurable space and let $Y : (\Omega, \mathscr{F}) \to (E, \xi)$ be a random variable. Then we say that a Markov kernel $$(Y|\mathscr{A}) : \Omega \times \xi \to [0,1]$$ on $(\Omega, \mathscr{A})$ and $(E, \xi)$ is a \textbf{regular conditional distribution} if for all $B \in \xi$, $$(Y|\mathscr{A})(\omega, B) = \mathbb{P}[Y \in B \,|\, \mathscr{A}](\omega) = \mathbb{E}[1_{Y^{-1}(B)}\,|\, \mathscr{A}](\omega).$$ In other words, $$\mathbb{E}[(Y|\mathscr{A})(\cdot, B)1_A] = \mathbb{P}[A \cap \{ Y \in B \}]$$ for all $A \in \mathscr{A}$. 
\end{definition}

Regular conditional distributions exists whenever $(\Omega, \mathscr{F}, \mathbb{P})$ is a standard probability space. To give a construction, recall (c.f. section 3 of \cite{EinsiedlerWardErgodicBook}) that there exist conditional measures $\mathbb{P}_{\omega}^{\mathscr{A}}$ uniquely characterized by $$\E[f|\mathscr{A}](\omega) = \int f \, d\mathbb{P}_{\omega}^\mathscr{A}.$$ Then $$(Y|\mathscr{A})(\omega, \cdot) = Y_{*}\mathbb{P}_{\omega}^\mathscr{A}$$ Indeed, for $B \in \xi$, $$(Y|\mathscr{A})(\omega, B) = E[1_{Y^{-1}(B)}|\mathscr{A}](\omega) = \int 1_{Y^{-1}(B)} \, d\mathbb{P}_{\omega}^\mathscr{A} = \mathbb{P}_{\omega}^\mathscr{A}(Y^{-1}(B)) = Y_{*}\mathbb{P}_{\omega}^\mathscr{A}(B).$$ We denote by $[Y|\mathscr{A}]$ a random variable defined on a separate probability space with law $(Y|\mathscr{A})$. 

We recall that given two further $\sigma$-algebras $\mathscr{G}_1, \mathscr{G}_2 \subset \mathscr{F}$, we say that they are independent given $\mathscr{A}$ if for all $U \in \mathscr{G}_1$ and $V\in \mathscr{G}_2$ $$\mathbb{P}[U \cap V|\mathscr{A}] = \mathbb{P}[U|\mathscr{A}]\mathbb{P}[V|\mathscr{A}]$$ almost surely. Similarly, two random variables $Y_1$ and $Y_2$ are independent given $\mathscr{A}$ if the $\sigma$-algebra they generate are. Note that if $Y_1$ is $\mathscr{A}$-measurable, then it is independent given $\mathscr{A}$ to every random variable $Y_2$. 

Given a topological group $G$ and two measures $\mu_1$ and $\mu_2$ we recall that the convolution $\mu_1 * \mu_2$ is defined as $$(\mu_1*\mu_2)(B) = \int\int 1_B(gh) \, d\mu_1(g)d\mu_2(h)$$ for any measurable set $B\subset G$.

\begin{lemma}\label{ConditionalConvolution}
    Let $(\Omega, \mathscr{F}, \mathbb{P})$ be a probability space, $G$ be a topological group and $g,h$ be $G$-valued random variables. Let $\mathscr{A} \subset \mathscr{F}$ be a $\sigma$-algebra and assume that $g$ and $h$ are independent given $\mathscr{A}$. Then the following properties hold:
    \begin{enumerate}[label = (\roman*)]
        \item $(gh|\mathscr{A}) = (g|\mathscr{A}) * (h|\mathscr{A})$ almost surely.
        \item $[gh|\mathscr{A}] = [g|\mathscr{A}]\cdot [h|\mathscr{A}]$ almost surely.
    \end{enumerate}
\end{lemma}

\begin{proof}
    To show (i), we note that by assumption $g$ and $h$ are independent with respect to $\mathbb{P}_{\omega}^{\mathscr{A}}$ for almost all $\omega \in \Omega$. This implies that for $f: G \to \R$ a continuous compactly supported function, $$\E_{\mathbb{P}_{\omega}^{\mathscr{A}}}[f(gh)] = \E_{\mathbb{P}_{\omega}^{\mathscr{A}}}[\E_{\mathbb{P}_{\omega}^{\mathscr{A}}}[f(gh)|h]] = \E_{(z_1,z_2) \sim \mathbb{P}_{\omega}^{\mathscr{A}} \times \mathbb{P}_{\omega}^{\mathscr{A}}}[f(g(z_1)h(z_2))],$$ proving (i). (ii) follows from (i) on a suitable separate probability space. 
\end{proof}

The aim of this subsection is to prove an abstract result relating entropy between scales and the trace. To do so, we first discuss conditional entropy and conditional trace. Let $Y$ be a random variable on a probability space $(\Omega, \mathscr{F}, \mathbb{P})$ and $\mathscr{A} \subset \mathscr{F}$ be a $\sigma$-algebra. Denote by $(Y|\mathscr{A})$ the regular conditional distribution as defined in section~\ref{RegConSection}. Assuming that $(Y|\mathscr{A})$ is almost surely absolutely continuous, we define $$H((Y \,|\, \mathscr{A}))(\omega) = H((Y|\mathscr{A})(\omega)).$$ 

Recall that if $X_1$ and $X_2$ are two random variables then entropy of $X_1$ given $X_2$ is $H(X_1|X_2) = H(X_1, X_2) - H(X_2).$ If $X_1$ and $X_2$ have finite entropy and finite joint entropy, then by \cite{Vignaux2021}*{Proposition 3},
\begin{equation}\label{ConditionalVarianceExpectation}
    H(X_1|X_2) = \mathbb{E}[H((X_1|X_2))].
\end{equation}

\subsection{Conditional Entropy}  \label{FirstEntropy:Conditional}

We next give an abstract definition of the entropy at a scale for a smoothing functions $s$. Indeed, let $g$ and $s$ be random variables on $G$ and assume that $s$ is absolutely continuous. Then the entropy at scale $s$ is defined as $$H(g; s) = H(gs) - H(s)$$ Moreover, if $s_1$ and $s_2$ are absolutely continuous smoothing functions we define the entropy between scales $s_1$ and $s_2$ as $$H(g; s_1|s_2) = H(g; s_1) - H(g; s_2).$$ The following basic result on the growth of conditional entropy holds.

\begin{lemma}\label{BasicEntropyGrowth}
    Let $g, s_1, s_2$ be independent random variables taking values in $G$. Assume that $s_1$ and $s_2$ are absolutely continuous with finite differential entropy and assume that $gs_1$ and $gs_2$ also have finite differential entropy. Then $$H(gs_1|gs_2) \geq H(g; s_1|s_2) + H(s_1).$$
\end{lemma}

\begin{proof}
    Note that $$H(gs_2|gs_1) \geq H(gs_2|g,s_1) = H(gs_2|g)  =  H(s_2),$$ having used in the inequality that conditioning reduces entropy as in \cite{CoverThomasBook}*{Section 8.6}, and so $$H(gs_2,gs_1) = H(gs_2|gs_1) + H(gs_1) \geq H(gs_1) + H(s_2).$$ Therefore 
    \begin{align*}
      H(gs_1|gs_2) &= H(gs_2,gs_1) - H(gs_2) \\ &\geq H(gs_1) - H(gs_2) + H(s_2) \\ &\geq H(g; s_1|s_2) + H(s_1). 
      \qedhere
      \end{align*}
\end{proof}

    \section{Entropy Growth for Stopped Random Walks}
    \label{EntropyGrowth}

The reader may recall the definition of $H_a(g;r)$ as given in \eqref{Hadef}. In this section we prove Theorem~\ref{StoppedRWHighEnt}, which we restate for convenience, and Corollary~\ref{CorStoppedRWHighEnt}. 

Given a probability measure $\mu$, we sample independent $\mu$-distributed random variables $\gamma_1, \gamma_2, \ldots $ that are independent from $(s_{a,r})_{a \geq 1, r > 0}$ and we denote $$q_n = \gamma_1\cdots \gamma_n.$$

\begin{theorem}\label{2ndStoppedRWHighEnt}(Theorem~\ref{StoppedRWHighEnt}) Let $\mu$ be a finitely supported probability measure on $G$. Let $\eta = (\eta_n)_{n \geq 1}$ be a sequence of stopping times satisfying the large deviation principle and write $L_n = \mathbb{E}[\eta_n]$ for $n\geq 1$. Let $a \geq 1$, $\eps > 0$ and let $r_n > 0$ be a sequence satisfying for all $n\geq 1$, $$r_n \leq a^{-1}c_G M_{\lceil (1 + \eps)L_n \rceil}$$ for a constant $c_G > 0$ depending only on $G$. Then for all $n \geq 1$,
	$$H_a(q_{\eta_{n}} ; r_n) \geq h_{\mu}\cdot L_n + o_{\mu,\eta,\eps}(L_n).$$
\end{theorem}

Recall that $H_a(q_{\eta_n}; r_n) = H(q_{\eta_n}s_{a, r_n}) - H(s_{a, r_n})$. To give the proof idea, note that by assuming a large deviation principle, with high probability $\eta_n \approx \E[\eta_n]$. Also, by definition of $h_{\mu}$, we have that $H(q_{L_n}) \geq h_{\mu} \cdot L_n$. On the other hand,  $s_{a, r_n}$ is mostly contained in a ball around the identity with radius $O(M_{L_n})$, and therefore by Lemma~\ref{EntProd} we have $H(q_{L_n}\cdot s_{a, r_n}) = H(q_{L_n}) + H(s_{a, r_n}),$ which implies the claim. We proceed with a more rigorous proof.

\begin{proof}
    We note that if the assumption holds for some $\eps$, it holds for all sufficiently small $\eps$. Therefore it suffices to show that for all sufficiently small fixed $\eps > 0$ we have that $$H_a(q_{\eta_{n}} ; r_n) \geq h_{\mu}\cdot L_n + O_{\mu,\eta}(\eps L_n) + o_{\mu,\eta,\eps}(L_n).$$ So fix some $\eps > 0$ which is sufficiently small in terms of  $\mu$ and consider $\eta_n'$ as $$\eta_n' = \begin{cases}
        \lceil (1 + \eps) L_n \rceil & \text{if } \eta_n > \lceil (1 + \eps) L_n \rceil, \\
        \lfloor (1 - \eps) L_n \rfloor & \text{if } \eta_n < \lfloor (1 - \eps) L_n \rfloor, \\
        \eta_n & \text{otherwise}.
    \end{cases}$$
    For a random variable $X$ denote by $\mathcal{L}(X)$ its law. Furthermore, given an event $A$, we will denote by $\mathcal{L}(X)|_A$ the measure given by the push forward of the restriction of $\mathbb{P}$ to $A$ under the random variable $X$. Note that $\|\mathcal{L}(X)|_A\,\|_1 = \mathbb{P}[A]$.

    By applying Lemma~\ref{EntSum1}, 
    \begin{align}
        H(q_{\eta_n}s_{a, r_n}) &= H(\mathcal{L}(q_{\eta_n}) * \mathcal{L}(s_{a, r_n})) \nonumber \\
        &\geq H(\mathcal{L}(q_{\eta_n})|_{\eta_n = \eta_n'} * \mathcal{L}(s_{a, r_n}))  + H(\mathcal{L}(q_{\eta_n})|_{\eta_n \neq \eta_n'} * \mathcal{L}(s_{a, r_n})) \nonumber \\
        &\geq H(\mathcal{L}(q_{\eta_n})|_{\eta_n = \eta_n'} * \mathcal{L}(s_{a, r_n})) + \mathbb{P}[\eta_n \neq \eta_n']H(\mathcal{L}(s_{a, r_n})), \label{EntScaleGrowth1FirstBound}
    \end{align} having used that  
    \begin{align*}
        H(\mathcal{L}(q_{\eta_n})|_{\eta_n \neq \eta_n'} * \mathcal{L}(s_{a, r_n}))  &\geq \mathbb{P}[\eta_n \neq \eta_n']H(\mathcal{L}(s_{a, r_n})),
    \end{align*} which can be shown by conditioning on $q_n$ (as conditioning reduces entropy \cite{CoverThomasBook}*{Section 8.6}) and using that $q_n$ and $s_{a,r_n}$ are independent.

    We next apply that $s_{a, r_n}$ has small support. Indeed, as $d(s_{a, r_n}, e) \ll_G r_na = o(M_{\lceil L_n(1 + \eps) \rceil})$ it follows that if $n$ is sufficiently large, $$d(s_{a, r_n}, \mathrm{Id}) <  \frac{1}{2}M_{\lceil L_n(1 + \eps) \rceil} \leq  \frac{1}{2}  \min_{x,y \in \mathrm{supp}(q_{\eta_n'}), x \neq y} d(x,y).$$ In particular, by Lemma~\ref{EntProd}, 
    \begin{equation}\label{EntScaleGrowth1SepEq}
        H(\mathcal{L}(q_{\eta_n})|_{\eta_n = \eta_n'} * \mathcal{L}(s_{a, r_n})) = H(\mathcal{L}(q_{\eta_n})|_{\eta_n = \eta_n'}) + \Prob[\eta_n = \eta_n'] H(\mathcal{L}(s_{a, r_n})).
    \end{equation}
    Combining \eqref{EntScaleGrowth1SepEq} with \eqref{EntScaleGrowth1FirstBound},  
    $$H(q_{\eta_n}s_{a, r_n}) \geq H(\mathcal{L}(q_{\eta_n})|_{\eta_n = \eta_n'}) + H(s_{a, r_n}).$$

    It remains to estimate $H(\mathcal{L}(q_{\eta_n})|_{\eta_n = \eta_n'})$. Consider the random variable $$X' = (q_{\lfloor (1-\eps) L_n \rfloor}, \gamma_{\lfloor (1-\eps) L_n \rfloor + 1}, \gamma_{\lfloor (1-\eps) L_n \rfloor + 2}, \ldots , \gamma_{\lceil (1+\eps) L_n \rceil}).$$ As $q_{\eta_n'}$ is completely determined by $X'$, we have $H(X'|q_{\eta_n'}) = H(X') - H(q_{\eta_n'}).$

    Let $K$ be the number of points in the support of $\mu$. Note that if $$\gamma_{\lfloor (1-\eps) L_n \rfloor + 1}, \gamma_{\lfloor (1-\eps) L_n \rfloor + 2}, \ldots , \gamma_{\lceil (1+\eps) L_n \rceil }$$ and $\eta_n'$ are fixed, then for any possible value of $q_{\eta_n'}$ there is at most one choice of $q_{\lfloor (1-\eps)L_n \rfloor}$ which would lead to this value of $q_{\eta_n'}$. Therefore for each $y$ in the image of $q_{\eta_n'}$ there are at most $(2\eps L_n + 2)K^{2\eps L_n + 2}$ elements $x$ in the image of $X'$ such that $\mathbb{P}[X' = x \cap q_{\eta_n'} = y] > 0$. Therefore $(X'|q_{\eta_n'})$ is almost surely supported on less than $(2\eps L_n + 2)K^{2\eps L_n + 2}$ points and hence by \eqref{ConditionalVarianceExpectation}, $$H(X'|q_{\eta_n'}) \leq \log\left((2\eps L_n + 2)K^{2\eps L_n + 2} \right) \leq  2\eps \log K \cdot L_n  + o_{\mu,\eps}( L_n). $$
    On the other hand, 
    \begin{equation}\label{HXprimebound}
        H(X') \geq H(q_{\lfloor L_n \rfloor}) \geq h_{\mu} \cdot \lfloor L_n \rfloor
    \end{equation}
    and therefore $$H(q_{\eta_n'}) \geq (h_{\mu} - 2\eps \log K) L_n - o_{\mu,\eps}(L_n). $$

    To continue, we note that by Lemma~\ref{EntSum2}, 
    \begin{equation}\label{SecondBound}
        H(q_{\eta_n'}) \leq H(\mathcal{L}(q_{\eta_n'})|_{\eta_n = \eta_n'}) + H(\mathcal{L}(q_{\eta_n'})|_{\eta_n \neq \eta_n'}) + \log 2.
    \end{equation} We wish to bound $H(\mathcal{L}(q_{\eta_n'})|_{\eta_n = \eta_n'})$ from below.  By the large deviation principle, $\Prob[\eta_n \neq \eta_n'] \leq \alpha^{L_n}$ for $\alpha \in (0,1)$ only depending on $\eps$ and $\mu$. We also know that conditional on $\eta_n \neq \eta_n'$, there are at most $2K^{\lceil (1 + \eps)m \rceil}$ possible values for $q_{\eta_n'}$ and therefore $$H(\mathcal{L}(q_{\eta_n'})|_{\eta_n \neq \eta_n'}) \leq \alpha^{L_n} \log\left( 2K^{\lceil (1 + \eps)L_n \rceil} \right) = o_{\mu,\eps}(L_n).  $$ This implies $$H(\mathcal{L}(q_{\eta_n'})|_{\eta_n = \eta_n'}) \geq (h_{\mu} - 2\eps \log K) L_n - o_{\mu, \eta, \eps}(L_n),$$ concluding the proof.
\end{proof}

\begin{proof}(of Corollary~\ref{CorStoppedRWHighEnt})
    For sufficiently small $\eps > 0$ it holds that as $L_n \to \infty$ that $$e^{-S\cdot L_n} = o(e^{-S_{\lceil (1 + \eps)L_n \rceil} \cdot \lceil (1 + \eps)L_n \rceil}) = o(M_{\lceil (1 + \eps)L_n \rceil})$$ as $S > (1 + \eps) S_{m}$ for all sufficiently large $m$. The claim follows from Theorem~\ref{2ndStoppedRWHighEnt}. 
\end{proof}

    \section{Entropy and Trace on Lie groups}   \label{EntropyTraceSection}

In this section we prove \eqref{IntroGVarIneq} in section~\ref{Entropy:EntTraceComp}  and in section~\ref{Entropy:BetweenScales} we establish Theorem~\ref{MainEntropyIncrease}. In section~\ref{Entropy:Conditional} an auxiliary result necessary for the proof of Theorem~\ref{MainEntropyIncrease}  will be shown.

\subsection{Entropy and Trace} \label{Entropy:EntTraceComp}

In this subsection we prove \eqref{IntroGVarIneq}. Recall that given $g_0 \in G$ and a random variable $g$ on $G$ we define $$\mathrm{tr}_{g_0}(g) = \mathrm{tr}(\log(g_0^{-1}g)),$$ whenever $\log(g_0^{-1}g)$ is defined.

\begin{proposition}\label{GVarianceEntIneq}
    Let $G$ be a Lie group of dimension $\ell$. Let $\eps > 0$ and suppose that $g$ is a absolutely continuous random variable taking values in  $B_{\eps}(g_0)$ for some $g_0 \in G$. If $\eps$ is sufficiently small depending on $G$, $$H(g) \leq \frac{\ell}{2} \log\left(  \frac{2\pi e}{\ell} \cdot \tr_{g_0}(g) \right) + O_G(\eps).$$
\end{proposition}

\begin{proof}
    We first note that if $X$ is an $\ell$-dimensional absolutely continuous random vector, then 
    \begin{equation}\label{lVarEntIneq}
        H(X) \leq \frac{\ell}{2}\log\left(\frac{2\pi e}{\ell}\cdot \mathrm{tr}(X) \right)
    \end{equation}    Indeed, it follows from the 1-dimensional case \eqref{1VarEntInequality} that $H(X) \leq \frac{1}{2}\log((2\pi e)^{\ell} \cdot |\mathrm{Var}(X)|),$ where $|\mathrm{Var}(X)|$ is the determinant of the covariance matrix. Note that by the AM-GM inequality $|\mathrm{Var}(X)| \leq \tr(X)^{\ell}\ell^{-\ell}$, which implies \eqref{lVarEntIneq}. 

    Since $H(g_0^{-1}g) = H(g)$ and $\tr_{g_0}(g) = \tr_e(g_0^{-1}g)$, we may assume without loss of generality that $g_0 = e$. The density $\frac{d\Haarof{G}|_{B_{\eps}(e)}}{d(\Haarof{\mathfrak{g}} \circ \log)|_{B_{\eps}(e)}}$ is smooth and for $\varepsilon > 0$ sufficiently small is $1 + O_G(\eps)$ and therefore $\sup \big|\log \frac{d\Haarof{G}|_{B_{\eps}(e)}}{d(\Haarof{\mathfrak{g}} \circ \log)|_{B_{\eps}(e)}}\big|  \ll_G \eps.$ Thus by Lemma~\ref{KLBound}, 
    \begin{align*}
        |D_{\mathrm{KL}}(g\,||\, \Haarof{G}) - D_{\mathrm{KL}}(g\,||\, \Haarof{\mathfrak{g}} \circ \log)| \ll_G \eps.
    \end{align*}
    The claim follows since by \eqref{lVarEntIneq} 
    \begin{align*}
        D_{\mathrm{KL}}(g\,||\, \Haarof{\mathfrak{g}} \circ \log) = D_{\mathrm{KL}}(\log(g) \,||\, m_{\mathfrak{g}}) = H(\log(g)) \leq \frac{\ell}{2} \log\left(  \frac{2\pi e}{\ell} \tr_{e}(g) \right).
    \end{align*}
\end{proof}

\subsection{Conditional Trace and Entropy between Scales}  \label{Entropy:Conditional}

We next define the conditional trace of a random variable on $G$ and relate it to the entropy between scales. 

\begin{definition}
    Let $g$ be a random variable defined on a probability space $(\Omega, \mathscr{F}, \mathbb{P})$ and taking values in $G$. Let $\mathscr{A} \subset \mathscr{F}$ be a $\sigma$-algebra let $g_0$ be a $\mathscr{A}$-measurable random variable taking values on $G$. Then we denote by $\tr_{g_0}(g\,|\, \mathscr{A})$ the $\mathscr{A}$-measurable function given for $\omega \in \Omega$ by $$\tr_{g_0}(g\,|\, \mathscr{A})(\omega) = \tr_{g_0(\omega)}((g\,|\, \mathscr{A})(\omega)),$$ whenever this expression is well-defined.
\end{definition}

We note here that the variance of a measure $\mu$ is defined as the variance of a random variable with law $\mu$. It follows from Proposition~\ref{GVarianceEntIneq} that when $(g|\mathscr{A})$ is almost surely absolutely continuous, 

\begin{equation}\label{ConditionalGVarianceEntIneq}
    H((g|\mathscr{A})) \leq \frac{\ell}{2} \log\left(  \frac{2\pi e}{\ell} \cdot \tr_{g_0}(g|\mathscr{A}) \right) + O_G(\eps).
\end{equation}

\begin{proposition}\label{AbstractEntropyIncrease}
    Let $g,s_1$ and $s_2$ be independent absolutely continuous random variables taking values in $G$ and suppose that that $s_1$ and $s_2$ are supported on $B_{\eps}$ for some sufficiently small $\eps > 0$ and have finite differential entropy. Write $c = \frac{\ell}{2} \log \frac{2\pi e}{\ell} \tr_e(s_1) - H(s_1)$ and suppose that $\tr_e(s_1) \geq A\eps^2$ for some constant $A > 0$. Then $$\mathbb{E}[\tr_{gs_2}(g|gs_2)] \geq \frac{2}{\ell}(H(g; s_1|s_2) - c - O_G(A^{-1}\eps))\tr_e(s_1).$$ 
\end{proposition}

We first prove some basic result on the trace of the product of two random variables. 

\begin{lemma}\label{GVarProdBound}
    Let $\eps > 0$ be sufficiently small and let $a,b$ be random variables and $\mathscr{A}$ a $\sigma$-algebra. Suppose that $b$ is independent from $a$ and $\mathscr{A}$ and let $g_0$ be an $\mathscr{A}$-measurable random variable. Suppose that $g_0^{-1}a$ and $b$ are almost surely contained in $B_{\eps}$. Then $$\tr_{g_0}(ab|\mathscr{A}) = \tr_{g_0}(a|\mathscr{A}) + \tr_{e}(b) + O_G(\eps^3).$$
\end{lemma} 

Note that under the assumptions of Lemma~\ref{GVarProdBound} it holds by Lemma~\ref{ConditionalConvolution} that $$[ab|\mathscr{A}] = [a|\mathscr{A}][b|\mathscr{A}] = [a|\mathscr{A}]b.$$ Therefore the claim follows from the following unconditional version.

\begin{lemma}
    Let $\eps > 0$ be sufficiently small and let $g$ and $h$ be independent random variables taking values in $G$. Suppose that the image of $g$ is contained in $B_{\eps}$ and the image of $h$ is contained in $B_{\eps}(h_0)$ for some $h_0 \in G$. Then $$\tr_{h_0}(hg) = \tr_{h_0}(h) + \tr_e(g) + O_G(\eps^3).$$
\end{lemma}

\begin{proof}
    Let $X = \log(h_0^{-1}h)$ and let $Y = \log(g)$. Then $|X|, |Y| \leq \eps$ almost surely and by Taylor's theorem there is a random variable $E$ with $|E| \ll \eps^2$ almost surely such that $$\log(\exp(X)\exp(Y)) = X + Y + E.$$ Therefore 
    \begin{align*}
        \tr_{h_0}(hg) &= \E[|X + Y + E|^2] - |\E[X + Y + E]|^2 \\
        &=  \E[|X + Y|^2] - |\E[X + Y]|^2 \\ &+ 2 \mathbb{E}[(X + Y) \cdot E] + \mathbb{E}[|E|^2] - 2\mathbb{E}[X + Y]\mathbb{E}[E] - |\mathbb{E}[E]|^2\\
        &= \mathrm{Var}[X + Y] + O_G(\eps^3) =  \mathrm{Var}[X] + \mathrm{Var}[Y] + O_G(\eps^3).
    \end{align*}
\end{proof}

\begin{proof}(of Proposition~\ref{AbstractEntropyIncrease})
    We note that by \eqref{ConditionalVarianceExpectation} and Lemma~\ref{BasicEntropyGrowth}, it holds that $$\E[H((gs_1 | gs_2))]  \geq H(g; s_1|s_2) + H(s_1)$$ and so by \eqref{ConditionalGVarianceEntIneq},
    $$\E\left[ \frac{\ell}{2} \log \frac{2\pi e}{\ell} \tr_{gs_2}(gs_1|gs_2)   \right] + O_G(\eps) \geq H(g; s_1|s_2)  + H(s_1).$$

    Note that $(gs_2)^{-1}g = s_2^{-1}$, which is contained in $B_{\eps}(e)$. Therefore by Lemma~\ref{GVarProdBound}, $$\tr_{gs_2}(gs_1|gs_2) \leq \tr_{gs_2}(g|gs_2) + \tr_e(s_1) + O_G(\eps^3)$$ and so $$H(g; s_1|s_2)  + H(s_1) \leq \E\left[ \frac{\ell}{2} \log \frac{2\pi e}{\ell}\left(\tr_{gs_2}(g|gs_2) + \tr_e(s_1) + O_G(\eps^3) \right)  \right] + O_G(\eps).$$ Thus
    $$\frac{2}{\ell}\left(  H(g; s_1|s_2) - c \right) \leq \E\left[ \log \left( 1 + \frac{\tr_{gs_2}(g|gs_2)}{\tr_e(s_1)}  + O_G(A^{-1}\eps)   \right) \right].$$ Using that  $\log(1 + x)\leq x$ for $x\geq 0$, we conclude the claim. 
\end{proof}

\subsection{Proof of Theorem~\ref{MainEntropyIncrease}} \label{Entropy:BetweenScales}

The proof relies on the following lemma. We  recall from the introduction that  $\beta_{a,r}$ is the random variable with density function $f_{a,r}: \mathfrak{g} \to \R$ given by $$f_{a,r}(x) = \begin{cases}
    C_{a,r}e^{-\frac{|x|^2}{2r^2}} & \text{if } |x| \leq ar, \\
    0 &\text{otherwise},
    \end{cases}$$ where $C_{a,r}$ is a normalising constant to ensure that $f_{a,r}$ integrates to $1$.

\begin{lemma}\label{VarEntProperties}
    The following properties hold for $r > 0$ and $a\geq 1$:
    \begin{enumerate}[label = (\roman*)]
        \item $\ell r^2 \ll \tr(\beta_{a,r}) \leq \ell r^2$ and for $\eta_{a,r}$ the distribution of $\beta_{a,r}$, $$H(\eta_{a,r}) = \frac{\ell}{2}  \log 2\pi e r^2 + O_{\ell}(e^{-a^2/4}).$$
        \item If $ar$ is sufficiently small,  $\ell r^2 \ll \tr_e(s_{a,r}) \leq \ell r^2$ and 
        $$H(s_{a,r}) = \frac{\ell}{2}  \log 2\pi e r^2 + O_{\ell}(e^{-a^2/4}) + O_{G}(ar).$$
    \end{enumerate}
\end{lemma}

\begin{proof}
    We note that (ii) follows from (i). To prove (i), we deal initially with the $r = 1$ case. Note first that $$\int_{x \in \R^{\ell}, |x| \leq a} e^{-|x|^2/2} \, dx \leq \int_{x \in \R^{\ell}} e^{-|x|^2/2} \, dx = \prod_{i = 1}^\ell \int_{\R} e^{-x_i^2/2 } \, dx_i = (2\pi)^{\ell/2}$$ and by using spherical coordinates 
    \begin{align*}
        \int_{x \in \R^{\ell}, |x| \geq a} e^{-|x|^2/2} \, dx &= c_{\ell} \int_{a}^{\infty} u^{\ell - 1} e^{-u^2/2} \, du \\ &\ll_{\ell} \int_{a}^{\infty}  e^{-u^2/3} \, du \leq \int_a^{\infty} e^{-au/3} \, du = \frac{3}{a}e^{-a^2/3} \ll_{\ell} e^{-a^2/4}.
    \end{align*} Thus we conclude $$\int_{x \in \R^{\ell}, |x| \leq a} e^{-|x|^2/2} \, dx  = (2\pi)^{\ell/2} - \int_{x \in \R^{\ell}, |x| \geq a} e^{-||x||^2/2} \, dx \geq (2\pi)^{\ell/2} - O_{\ell}(e^{-a^2/4})$$ and therefore $C_{1,a} = (2\pi)^{-\ell/2} + O_{\ell}(e^{-a^2/4}).$ We are now in a suitable position to calculate $H(\eta_{1,a})$. Indeed, 
    \begin{align*}
        H(\eta_{1,a}) &= \int_{|x| \leq a} -C_{1,a} e^{-|x|^2/2} \log\left(C_{1,a} e^{-|x|^2/2} \right) \, dx \\ &= \int_{|x| \leq a} C_{1,a}\left(  \frac{|x|^2}{2} - \log C_{1,a} \right) e^{-|x|^2/2}  \, dx 
    \end{align*} 
    We calculate 
    \begin{align*}
    &\int_{x\in \R^{\ell}} C_{1,a}\left(  \frac{|x|^2}{2} - \log C_{1,a} \right) e^{-|x|^2/2}  \, dx \\ &= (2\pi)^{\ell/2} C_{1,a}\left( \frac{\ell}{2} - \log C_{1,a} \right) \\
    &= \left( 1 + O_{\ell}(e^{-a^2/4}) \right)\left(  \frac{\ell}{2}\log e + \frac{\ell}{2} \log 2\pi + O_{\ell}(e^{-a^2/4}) \right) \\
    &= \frac{\ell}{2} \log 2\pi e + O_{\ell}(e^{-a^2/4}).
    \end{align*} and again using spherical coordinates, 
    \begin{align*}
    &\int_{|x| \geq a} C_{1,a}\left(  \frac{|x|^2}{2} - \log C_{1,a} \right) e^{-|x|^2/2}  \, dx \\ &= c_{\ell}\int_{a}^{\infty} C_{1,a}\left(  \frac{u^2}{2} - \log C_{1,a} \right) u^{\ell - 1} e^{-u^2/2}  \, dx \\
    &\ll_{\ell} O_{\ell}(e^{-a^2/4}).
    \end{align*} Thus the claimed bound on $H(\eta_{1,a})$ follows. Since $f_{a,r}(x) = r^{\ell}C_{1,a}f_{1,a}(x/r)$ it follows that $H(\eta_{a,r}) = \log(r^{\ell}) + H(\eta_{1,a})$ and hence the proof is complete. 
\end{proof}

\begin{proof}(of Theorem~\ref{MainEntropyIncrease})
    We apply Proposition~\ref{AbstractEntropyIncrease} to $s_1 = s_{a,r}$ and $s_2 = s_{a, 2r}$ and we set $\eps = \ell ar$. By Lemma~\ref{VarEntProperties} (ii) we have that $\mathrm{tr}_e(s_{1}) \gg_{G} \ell r^2 \gg_{G} a^{-2}\eps^2$ and  $c = \frac{\ell}{2} \log \frac{2\pi e}{\ell} \mathrm{tr}_e(s_1) - H(s_1) \leq O_{G}(e^{-a^2/4} + ar).$ Applying Proposition~\ref{AbstractEntropyIncrease} with $A = a^{-2}$, $$\E[\mathrm{tr}_{gs_2}(g|gs_2)]  \gg_G  r^2(H_a(g; r|2r) - O_{G}(e^{-a^2/4} + a^3 r)).$$ On the other hand, we have that $|\log((gs_2)^{-1}g)| = |\log s_2| \leq 2ar$ and therefore $$\tr(g;2ar) \geq (2ar)^{-2}\E[\mathrm{tr}_{gs_2}(g|gs_2)]\gg_G a^{-2}(H_a(g; r|2r) - O_{G}(e^{-a^2/4} + a^3r)).$$
\end{proof}

    \section{From Entropy Gap to Trace Sum} \label{EntGapTrSumSection}

In this section we prove Proposition~\ref{EntropyGapToTrSum}, which we deduce from the following two propositions.

\begin{proposition}\label{trIntBound}
    Let $g$ be a $G$-valued random variable independent of $(s_{a,r})_{a \geq 1, r> 0}$ and let $0 < r_1 < r_2$. Let $a \geq 1$ such that $ar_2$ is sufficiently small in terms of $G$. Suppose that for all $r_1' \in [r_1, 2r_1]$ as well as $r_2' \in [r_2/2, 2r_2]$ it holds for some constant $C> 0$ that $$H_a(g; r_1'| r_2') \geq C.$$ Then $$\int_{ar_1}^{4ar_2} \frac{1}{u} \tr(g;u) \, du \gg_G a^{-2}(C - N\cdot (O_G(e^{-\frac{a^2}{4}} + a^3r_2))$$ for $N = \left\lceil \frac{\log r_2 - \log r_1}{\log 2} \right\rceil - 1$.
\end{proposition}

\begin{proof}
    Let $a \geq 1$ and set $ N = \left\lceil \frac{\log r_2 - \log r_1}{\log 2} \right\rceil - 1 $. Note that $2^{N+1} r_1 \geq r_2$ as well as $2^{N} r_1 \leq r_2$.
    
    Given $u \in [1,2)$ and an integer $1\leq i \leq N$ denote $$k_i(u) = H_a(g; 2^{i-1}ur_1| 2^i u r_1).$$ Then by Theorem~\ref{MainEntropyIncrease}, there is some constant $c = c(G) > 0$ depending only on $G$ such that 
    \begin{equation}\label{trbound}
        \mathrm{tr}(g; a 2^i u r_1) \geq c a^{-2}(k_i(u) - O_G(e^{-a^2/4} + a^3 2^i r_1) ).
    \end{equation} Thus $$\sum_{i = 1}^{N} \tr(g; a2^i u r_1) \geq ca^{-2} \left( \sum_{i = 1}^N k_i(u) - O_G(Ne^{-\frac{a^2}{4}} + Na^3 2^N r_1) \right).$$ Note that for $u \in [1,2)$ we have $a2^{N + 1} u r_1 \leq 4ar_2$ and $a u r_1 \geq a r_1$. Therefore, upon substituting $v = a2^{i}ur_1$ in the third line, \begin{align}
        &\int_{ar_1}^{4ar_2}  \frac{1}{v} \tr(g; v) \, dv \nonumber \\ \geq &\sum_{i = 1}^{N} \int_{a2^{i} u r_1}^{a2^{i + 1}u r_1} \frac{1}{v} \tr(g; v) \, dv \nonumber\\
        \geq &\sum_{i = 1}^{N} \int_{1}^{2} \frac{1}{u} \tr(g; a 2^i ur_1) \, du \nonumber \\
        \geq &ca^{-2} \int_1^2 \frac{1}{u}\left( \sum_{i = 1}^N k_i(u) - O_{G}(Ne^{-\frac{a^2}{4}} + N a^3 2^N r_1)\right) du. \label{IntBound}
    \end{align}

    Observe that by our assumption $\sum_{i = 1}^N k_i(u) = H_a(g; ur_1| 2^N ur_1 ) \geq C$  and therefore the claim follows using that $2^N r_1 \leq r_2$. 
\end{proof}

\begin{proposition}\label{LotsOfTrace}
    Suppose that for a $G$-valued random variable $g$ and $0 < r_1 < r_2$ we have for some constant $C_1 > 0$ that $$\int_{r_1}^{r_2} \frac{1}{u} \tr(g; u) \, du \geq C_1.$$ Let $A > 1$. Then there exists $s_1, \dots, s_{m} \in (r_1, r_2)$ where $m = \lceil \frac{\log r_2 - \log r_1}{2\log A} \rceil$ such that
    \begin{equation*}
        \sum_{i=1}^m \tr (g ; s_i) \geq \frac{C_1}{4 \log A} \quad\quad\quad \text{ and } \quad\quad\quad s_{i+1} \geq A s_i
    \end{equation*}
    for all $1 \leq i \leq m-1$.
\end{proposition}

\begin{proof}
    Define $a_1, a_2, \ldots , a_{2m + 1}$ by $a_i = r_1A^{i-1}$. Therefore $a_1 = r_1$ and $a_{2m + 1} \geq r_2$. Let $U$ and $V$ be defined by $$U = \bigcup_{i = 1}^{m} [a_{2i-1}, a_{2i}) \quad \text{  and  } \quad V = \bigcup_{i = 1}^{m} [a_{2i}, a_{2i+1}).$$
    Without loss of generality, upon replacing $U$ with $V$, by our assumption  $$\int_U \frac{1}{u} \tr(g; u) \, du \geq C_1/2.$$
    For $i \in [m]$ let $s_i \in (a_{2i-1}, a_{2i})$ be chosen such that $$\tr(g; s_i) \geq \frac{1}{2}\sup_{u \in (a_{2i-1}, a_{2i})} \tr(g; u).$$ In particular, $$\tr(g; s_i) \geq \frac{1}{2 \log A } \int_{a_{2i-1}}^{a_{2i}} \frac{1}{u}\tr(g; u) \, du.$$ Summing over $i$ gives 
    \begin{align*}
        \sum_{i = 1}^{m} \tr(g; s_i) \geq \frac{1}{2 \log A} \int_U \frac{1}{u} \tr(g; u) \, du \geq \frac{C_1}{4 \log A}.
    \end{align*} 
\end{proof}

To deduce Proposition~\ref{EntropyGapToTrSum}, one uses Proposition~\ref{LotsOfTrace} with the range $(ar_1, 4ar_2)$ and with $$C_1 =  a^{-2}(C - N\cdot (O_G(e^{-\frac{a^2}{4}} + a^3r_2)).$$

\bibliography{referencesgeneral.bib}
\end{document}